\documentclass{amsart}

\usepackage{amssymb}

\usepackage{}

\newtheorem{theorem}{Theorem}[section]

\theoremstyle{definition}
\newtheorem{definition}[theorem]{Definition}
\newtheorem{example}[theorem]{Example}

\theoremstyle{remark}

\theoremstyle{proposition}
\newtheorem{proposition}[theorem]{Proposition}

\theoremstyle{corollary}
\newtheorem{corollary}[theorem]{Corollary}

\numberwithin{equation}{section}

\newcommand{\cpx}{{\mathbb C} }
\newcommand{\real}{{\mathbb R}}

\begin{document}

\title[Clebsch-Gordan coefficients]%
{A rational theory of Clebsch-Gordan coefficients} 


\author{Robert W. Donley, Jr.}
\address{Department of Mathematics and Computer Science,  Queensborough Community College (CUNY), Bayside, NY 11364, USA}
\curraddr{}
\email{rdonley@qcc.cuny.edu}
\thanks{The first author was supported by PSC-CUNY Research Awards Program 46 (Trad. A)}

\author{Won Geun Kim}
\address{Department of Mathematics, The City College of New York (CUNY), New York, NY 10031, USA}
\curraddr{}
\email{wkim@ccny.cuny.edu}
\thanks{}

\subjclass[2010]{Primary 22E70; Seconday 81R05}
\keywords{Clebsch-Gordan coefficient, Wigner coefficient, 3-j symbol, coupling coefficient,  angular momentum, Pascal's recurrence}
\dedicatory{Dedicated to Professor {\' O}lafsson on the occasion of his sixty-fifth birthday.}
\date{}

\begin{abstract}
A theory of Clebsch-Gordan coefficients for $SL(2, \mathbb{C})$ is given using only rational numbers.  Features include orthogonality relations, recurrence relations, and Regge's symmetry group. Results follow from elementary representation theory and properties of binomial coefficients.  A computational algorithm is given based on Pascal's recurrence.
\end{abstract}

\maketitle


\section{Introduction}

In Wigner's formula for Clebsch-Gordan coefficients for $SU(2)$ (also called Wigner coefficients; see also 3-j symbols), the main sum may be expressed with terms consisting of triple products of binomial coefficients.  Traditional approaches to the theory rely on special function methods:  hypergeometric series, Jacobi polynomials \cite{Vi}, and Hahn polynomials \cite{Ko}.  On the other hand, an approach to Racah's alternative formula is given in \cite{Ge}, 10.14--10.16 using mostly combinatorial methods, but still rooted in use of the inner product.  See also \cite{FR}, Ch.5, or \cite{Vi}, Ch. III, Sec. 8.9, (10).  

Several notions drive the present work:  that the tensor product problem exists independent of orthonormal bases,  that finite-dimensional representations of $SL(2,\cpx)$ are structurally rational, and that finite-dimensional representations are somewhat combinatorial.  Of course, the finite-dimensional representation theory of $SU(2)$ matches to that of $SL(2, \cpx)$ by Weyl's Unitary Trick.  Here a theory is presented with the following features:

\begin{enumerate}
\item an explicit formula for weight vectors, whose integral coordinates may be expressed through a generating function in two variables (Definitions 4.5, 4.6),
\item a summation formula for Clebsch-Gordan coefficients with rational values (Theorem 4.10, Corollary 8.5),
\item orthogonality relations (Proposition 4.4, Corollary 4.11),
\item an algorithm based on Pascal's recurrence for computer implementation (Theorem 6.1),
\item recurrence relations (Propositions 7.1, 7.2, 7.3),
\item a version of Regge's group of order 72 (Section 8), and
\item derivations of Wigner's and Racah's formulas (Section 9).
\end{enumerate}   

With a similar philosophy in mind, a projection operator method is given in Section 10.  The connection to Pascal's Triangle follows empirically from consideration of these projection operators. Additionally formula (10.4) corrects the scaling factor in the Casimir extremal projectors in \cite{Do},  Definitions 5.5 and 6.1, for the case of $SL(2, \real);$ further adjustments are needed for general groups.

Since the techniques used are algebraic, all the results in this work hold without changes for $SL(2, F)$, where $F$ is any algebraically closed field of characteristic zero. For instance, see \cite{Hu}, VI.22, Problem 7. In fact, results follow over a general field of characteristic zero; the Clebsch-Gordan decomposition follows from the highest weight assumption.

The authors thank the referee for useful comments and suggesting the comprehensive handbook \cite{VM}. The authors also thank Bent {\large $\varnothing$}rsted and Bart Van Steirteghem for helpful conversations. 

\section{Definitions and notation}

Let $SL(2, \cpx)$ be the linear semisimple Lie group of unimodular invertible matrices of size 2 over the complex numbers, let $sl(2)$ be the Lie algebra of traceless matrices of size 2 over $\cpx$, and let $U(sl(2))$ be the associated universal enveloping algebra.   Define the standard basis for $sl(2)$ by 
\begin{equation}e=\begin{bmatrix} 0 &1 \\ 0 & 0 \end{bmatrix}, \ f=\begin{bmatrix} 0 &0 \\ 1 & 0  \end{bmatrix}, \ h=\begin{bmatrix} 1 &0 \\ \ 0 & -1 \end{bmatrix}\end{equation}

Fix $n\ge 0,$ and let $V(n)$ denote the irreducible finite-dimensional representation for $SL(2, \cpx)$ with highest weight $n$ with respect to the $h$ and $e$.  In this case, $\text{dim}_\cpx\ V(n) = n+1,$ and if $\phi_n$ is a nonzero highest weight vector for $V(n)$, then the set 
\begin{equation} B = \{\phi_n,\  f\phi_n,\  f^2\phi_n,\  ...,\ f^n\phi_n\}\end{equation}

\noindent forms a basis for $V(n)$.  With respect to the Lie algebra action, one has 

\begin{enumerate}
\item[(LA.1)] $h\cdot f^i\phi_n = (n-2i)f^i\phi_n$,
\item[(LA.2)] $e\cdot f^i\phi_n = i(n-i+1)f^{i-1}\phi_n$, and
\item[(LA.3)] $f\cdot f^i\phi_n = f^{i+1}\phi_n$.  
\end{enumerate}
In particular,
\begin{enumerate}
\item[(LA.4)] $f^i\phi_n$ is a weight vector of weight $n-2i$,
\item[(LA.5)] $e\cdot \phi_n = 0$\qquad\qquad(highest weight vector with weight $n$), and
\item[(LA.6)] $f\cdot f^n\phi_n = 0$\qquad\ \ (lowest weight vector with weight $-n$).
\end{enumerate}

Note that, as a representation of $SL(2, \cpx)$, 
\begin{equation} V(n)\cong Sym^n(V(1)),\end{equation}
where $V(1)$ arises from the standard action of $SL(2, \cpx)$ on $\cpx^2.$
The space $V(1)$ admits a non-degenerate, invariant bilinear form by
\begin{equation}\langle (x_1, y_1), (x_2, y_2)\rangle = det \begin{bmatrix} x_1 & x_2 \\ y_1 & y_2\end{bmatrix}.\end{equation}

If $V$ and $W$ are representations of $SL(2, \cpx)$ with non-degenerate, invariant bilinear forms $\langle\cdot, \cdot\rangle_1$ and $\langle\cdot. \cdot\rangle_2$, we obtain a non-degenerate, invariant bilinear form on $V\otimes W$ by extending
\begin{equation}\langle v_1\otimes w_1, v_2\otimes w_2\rangle = \langle v_1, v_2\rangle_1 \ \langle w_1, w_2\rangle_2.\end{equation}
One may extend this definition to any finite number of factors in the tensor product.  In this way, one obtains a non-degenerate, invariant bilinear form on $Sym^n(V(1))$, unique up to a complex scalar.

For a given $V(n)$-type, any invariant form satisfies 
\begin{equation}\langle X\cdot v, w\rangle_{n} = \langle v, -X\cdot w\rangle_n\end{equation} for $X$ in $sl(2).$
Since
\begin{equation}\langle f\cdot v, w\rangle_n = \langle v, -f\cdot w\rangle_n,\end{equation}
rescaling the form such that 
\begin{equation}\langle \phi_n, f^n\phi_n\rangle_n = 1\end{equation}
determines a normalization for all basis vectors.
That is,
\begin{equation}\langle f^i\phi_n, f^{n-i}\phi_n\rangle_n = \langle \phi_n, (-1)^i f^n\phi_n\rangle_n = (-1)^i.\end{equation} and, among basis vectors, $f^i\phi_n$ pairs nontrivially only with $f^{n-i}\phi_n.$ 
Finally it follows that this form is symmetric when $n$ is even and alternating when $n$ is odd.

\section{Tensor products}

Notation for Clebsch-Gordan coefficients varies greatly in the literature, here presenting yet another variation.  In addition to dropping the traditional half-integer notation associated with $SU(2)$, this variation facilitates  an emphasis on exponents over weights as indices, an organizing principle based on matrices, and  a convenient indexing of a  five-dimensional polyhedron in subsequent work. Indices always satisfy 
\begin{equation} 0 \le i \le m,\quad 0\le j\le n,\quad 0\le k \le \min(m, n),\quad k\le i+j \le m+n-k.\end{equation} Unless stated otherwise, assume $m\le n$, so that $0\le k \le m \le n.$

One obtains an action of $SL(2, \cpx)$ on $V(m)\otimes V(n)$ by 
\begin{equation}g\cdot(v\otimes w) = gv\otimes gw\end{equation}
and an action of $sl(2)$ by

\begin{equation}X\cdot (v\otimes w) = Xv\otimes w + v\otimes Xw.\end{equation}
With these actions, one obtains

\begin{proposition}[Clebsch-Gordan Decomposition]  With respect to representations of $SL(2, \cpx)$, one has
\begin{equation}V(m)\otimes V(n) \cong V(n-m)\oplus V(n-m+2) \oplus \dots \oplus V(m+n)\end{equation}
\end{proposition}

\begin{proof}[Proof using Section 4] This proof works over a field of characteristic zero. Irreducible, finite-dimensional highest weight representations are determined by their weight sets as in LA.1-6.  Explicit highest and lowest weight vectors for each summand are given in Propositions 4.1 and 4.9. That is, each summand occurs as a subrepresentation of the tensor product, and counting dimensions gives
\begin{equation}\sum\limits_{s=0}^m (n-m+2s+1)=(m+1)(n+1).\end{equation}
Since each irreducible subrepresentation has an invariant complement, the sum is direct.
\end{proof}

The weight set of $V(m)\otimes V(n)$ is all pairwise sums of weights for $V(m)$ with weights for $V(n)$. In particular,
\begin{equation}h\cdot (f^i\phi_m\otimes f^j\phi_n) = (m-2i + n - 2j)\ f^i\phi_m\otimes f^j\phi_n\end{equation}
One immediately notes that
\begin{equation}{\bf B} = \{ f^i\phi_m\otimes f^j\phi_n\},\qquad (0\le i \le m,\ 0\le j\le n)\end{equation}
forms a basis of weight vectors for $V(m)\otimes V(n),$ but in general this basis does not respect the decomposition given above.  Furthermore, the subspace of weight vectors of weight $m+n-2p$ has basis
\begin{equation}{\bf B}_p = \{ f^i\phi_m\otimes f^j\phi_n\},\qquad (i+j=p).\end{equation}

In the classical case for $SU(2)$, one describes an orthonormal basis of weight vectors for each $V(m')$ in the decomposition of $V(m)\otimes V(n).$  Choose orthonormal bases $\{ u_i\}$ (where $u_i$ has weight $m-2i$) and $\{ v_j\}$ (where $v_j$ has weight $n-2j$) for $V(m)$ and $V(n)$, respectively. After setting conventions, one may define an orthonormal basis $\{ w_{m, n, k}(p)\}$ (weight $m+n-2p$) for each $V(m+n-2k)$ in the tensor product such that
\begin{equation}w_{m, n, k}(p) = \sum_{i+j=p}  c'_{m, n, k}(i, j)\ u_i\otimes v_j\end{equation}
with each $c'_{m, n, k}(i, j)$ real.

In turn, the Clebsch-Gordan coefficients are defined by
\begin{equation}u_i\otimes v_j = \sum_{k=0}^m\ C'_{m, n, k}(i, j) w_{m, n, k}(i+j).\end{equation}

After converting to the present notation for indices and changing $s$ to $i-s$ in the summation, one has Wigner's formula in binomial form:

\begin{proposition}[\cite{Vi}, Ch. III.8.3, Equation ($14^\prime$)]

 {\fontsize{10}{10}\selectfont
\begin{equation*}C'_{m,n,k}(i, j)\ =\sqrt{  \frac{(m+n-2k+1)\begin{pmatrix} m+n-k \\ m-i,\ n-j,\ *\end{pmatrix} }{{(m+n-k+1)\ \begin{pmatrix} m+n-k \\ m-k,\ n-k,\ k\end{pmatrix}  \begin{pmatrix} m+n-k \\ i,\ j,\ *\end{pmatrix}}}}\cdot S
\end{equation*}
}
where 
\begin{equation*}S=\sum\limits_{s=0}^k \ (-1)^{s} \begin{pmatrix} i+j-k \\ i-s\end{pmatrix}\begin{pmatrix} m-s \\ k-s\end{pmatrix}\begin{pmatrix} n-k+s \\ s\end{pmatrix}\end{equation*}
\end{proposition} 

\noindent   For non-negative integers $a, b , c$, multinomial coefficients are defined by

\begin{equation}\begin{pmatrix} a+b \\ a\end{pmatrix} = \frac{(a+b)!}{a!\ b!}\qquad\text{and}\qquad\begin{pmatrix} a+b+c \\ a,\ b,\ c\end{pmatrix} = \frac{(a+b+c)!}{a!\ b!\ c!}.\end{equation}

\noindent When notation is cumbersome, the lower index $c$ is replaced with $*$. In this work, these coefficients are set to zero if negative entries occur.

\section{Generating functions for weight vectors}

Among conventions, of key importance is a fixed choice of weight vectors. Here conventions are set by a uniform description of highest weight vectors.  This choice determines the weight vector bases for all $V(m')$-types that occur, and, by duality, fixes sign conventions for Clebsch-Gordan coefficients.

\begin{proposition}[Highest Weight Vector for $V(m+n-2k)$]  Define the highest weight vector $\phi_{m,n,k}$ of $V(m+n-2k)$ in $V(m)\otimes V(n)$ such that the coordinate vector $\lbrack\phi_{m, n, k}\rbrack_{\bf B_k}$ in $\cpx^{k+1}$ is given by

{\fontsize{9.7}{10}\selectfont
\begin{equation*}\biggl(\begin{pmatrix}m \\ k \end{pmatrix}\begin{pmatrix}n-k \\ 0 \end{pmatrix}, \dots, (-1)^l\begin{pmatrix}m-l \\ k-l \end{pmatrix}\begin{pmatrix}n-k+l \\ l \end{pmatrix}, \dots, (-1)^k\begin{pmatrix}m-k \\ 0 \end{pmatrix}\begin{pmatrix}n \\ k \end{pmatrix}\biggr).\end{equation*}
}
\end{proposition}

\begin{proof} See \cite{HT}, Sec. II, Prop. 2.1.1.  Assume the coordinate vector has the form
\begin{equation*}(c_0, c_1, c_2, \dots, c_k).\end{equation*}
with $c_0=\begin{pmatrix}m \\ k \end{pmatrix}.$   For each basis vector, applying (LA.2) gives
{\fontsize{9}{10}\selectfont
\begin{equation*}e\cdot (f^i\phi_m\otimes f^{k-i}\phi_n) = i(m-i+1) f^{i-1}\phi_m\otimes f^{k-i}\phi_n + (k-i)(n-k+i+1)f^i\otimes f^{k-i-1}\phi_n.\end{equation*}
}
\noindent Since $e\cdot \phi_{m,n,k} = 0,$ each $c_i$ may be computed from $c_{i-1}$. For instance,
\begin{equation*}k(n-k+1) \begin{pmatrix}m \\ k \end{pmatrix} + m\, c_1 = 0,\end{equation*}
\begin{equation*}(k-1)(n-k+2) c_1 + 2(m-1)\, c_2 = 0,\end{equation*}
and so on. Continuing in this manner, the proposition holds.
\end{proof}

Weight vectors for $V(m+n-2k)$ are obtained by repeated application of $f$ to $\phi_{m,n,k}$. Thus a weight vector basis for $V(m+n-2k)$ in $V(m)\otimes V(n)$ is given by
\begin{equation}{\bf B^k} = \{\phi_{m,n,k}, \ f\phi_{m,n,k}, f^2\phi_{m,n,k},\dots, f^{m+n-2k}\phi_{m,n,k}\}\end{equation}

\begin{definition}Define the coordinate functions $c_{m,n,k}(i, j)$ by the formula
\begin{equation}f^{p-k}\phi_{m,n,k} = \sum\limits_{i+j=p}\ c_{m,n,k}(i,j)\ f^i\phi_m\otimes f^j\phi_n.\end{equation}
\end{definition}

\begin{definition} Define the Clebsch-Gordan coefficients $C_{m,n,k}(i,j)$ by the formula
\begin{equation}f^i\phi_m\otimes f^j\phi_n = \sum\limits_{k=0}^m C_{m,n,k}(i, j)\ f^{i+j-k}\phi_{m,n,k}.\end{equation}
\end{definition}

It follows immediately from these definitions that

\begin{proposition}[Orthogonality Relations 1]
\begin{equation}\sum\limits_{i+j=p}\ c_{m,n,k}(i, j)\ C_{m,n,k'}(i, j) = \delta_{k, k'}\end{equation}
\begin{equation}\sum\limits_{k=0}^m\ c_{m,n,k}(i, j)\ C_{m, n, k}(i', j') = \delta_{i,i'}\delta_{j, j'}\quad (i+j=i'+j')\end{equation}
\end{proposition}

Next recall the binomial series for positive integer $r$
\begin{equation}\frac{1}{(1-x)^r} = \sum_{i=0}^\infty \begin{pmatrix} r-1+i \\ i \end{pmatrix} x^i.\end{equation}
One observes that the coordinates of the highest weight vector are products of terms from ascending and descending binomial series, which yields

\begin{definition}[Generating Function for Highest Weight Vectors] The coordinate $c_{m,n,k}(i, k-i)$ for $f^i\phi_m\otimes f^{k-i}\phi_n$ in $\lbrack \phi_{m,n,k}\rbrack_{\bf B_k}$ is the coefficient of $x^{k-i}y^{i}$ in 
\begin{equation}\frac{1}{(1-x)^{m-k+1}(1+y)^{n-k+1}}\end{equation}
\end{definition}

Noting that
\begin{equation}f\cdot (f^i\phi_m\otimes f^j\phi_n) = f^{i+1}\phi_m\otimes f^j\phi_n + f^i\phi_m\otimes f^{j+1}\phi_n,\end{equation}
coordinate vectors for these weight vectors are obtained by binomial recursion of coordinates (Pascal's Recurrence).  This effect translates immediately to generating functions.

\begin{definition}[Generating Function for Weight Vectors] Consider the weight vector $f^{p-k}\phi_{m,n,k}$ of weight $m+n-2p$ in the $V(m+n-2k)$-type. For $i+j=p$, the coordinate $c_{m,n,k}(i, j)$ of $f^i\phi_m\otimes f^j\phi_n$ in $f^{p-k}\phi_{m,n,k}$ is the coefficient of $x^jy^i$ in 
\begin{equation}\frac{(x+y)^{i+j-k}}{(1-x)^{m-k+1}(1+y)^{n-k+1}}\end{equation}
\end{definition}

To obtain an explicit formula similar to Wigner's formula, one expands the binomial series as an alternating sum of triple products of binomial coefficients.  In the following formulas, one uses the convention that binomial coefficients with negative entries are set to zero.

\begin{proposition}[Coordinate vector of weight $m+n-2i-2j$ in $V(m+n-2k)$]
\begin{equation}c_{m,n,k}(i,j) = \sum_{l=0}^{k} (-1)^l  \begin{pmatrix}i+j-k \\ i-l \end{pmatrix}\begin{pmatrix}m-l \\ k-l \end{pmatrix}\begin{pmatrix}n-k+l \\ l \end{pmatrix}\end{equation}
\end{proposition}

The corresponding formula for Clebsch-Gordan coefficients follows from a variant of the classical Chu-Vandermonde identity (\cite{AAR}, p. 67) and an explicit form for the lowest weight vectors.

\begin{proposition}  Suppose $r, s\ge0$ and $0\le k \le \min(r, s)$. Then
\begin{equation}\sum_{l=0}^k  (-1)^{l} \begin{pmatrix} k \\ l \end{pmatrix} \begin{pmatrix} r+s-k+l \\ s-k+l \end{pmatrix} =  (-1)^k \begin{pmatrix} r + s -k\\ s \end{pmatrix}.\end{equation}
\end{proposition}

\begin{proof} With $0\le k \le r,$ expand 
\begin{equation*}\frac{(1-x)^k}{(1-x)^{r+1}} = \frac{1}{(1-x)^{r-k+1}}\end{equation*}
into series
\begin{equation*}\sum_{l'=0}^k (-1)^{l'} \begin{pmatrix} k \\ l' \end{pmatrix}x^{l'}\ \cdot \sum_{l''=0}^\infty \begin{pmatrix} r+l'' \\ l'' \end{pmatrix}x^{l''} = \sum_{l=0}^\infty \begin{pmatrix} r-k+l \\ l \end{pmatrix} x^l.\end{equation*}
Comparing coefficients for $x^s$ yields
\begin{equation*}\sum_{l=0}^k  (-1)^{l} \begin{pmatrix} k \\ l \end{pmatrix} \begin{pmatrix} r+s-l \\ s-l \end{pmatrix} =  \begin{pmatrix} r-k + s \\ s \end{pmatrix}.\end{equation*}
Substituting $l'=k-l$ gives (4.11).
\end{proof}

\begin{proposition}[Lowest Weight Vector for $V(m+n-2k)$] The coordinate vector of the lowest weight vector $f^{m+n-2k}\phi_{m,n,k}$ of $V(m+n-2k)$ in $V(m)\otimes V(n)$ is given by
$$\lbrack f^{m+n-2k}\phi_{m, n, k}\rbrack_{\bf B_{m+n-k}}\  =\ \begin{pmatrix}m+n-2k \\ m-k \end{pmatrix} (1, -1, 1,\dots, (-1)^k)\quad\in \cpx^{k+1}.$$
\end{proposition}

\begin{proof} That $(1, -1, 1, \dots, (-1)^k)$ is a lowest weight vector follows in a manner similar to the highest weight vector formula, only now using the effect of $f$ on ${\bf B_{m+n-k}}.$ To find the scalar, the last entry is found using the coordinate vector formula with $i =m$ and $j=n-k.$ This gives

\begin{align*}c_{m,n,k}(m,n-k) &= \sum_{l=0}^{k} (-1)^l  \begin{pmatrix}m+n-2k \\ m-l \end{pmatrix}\begin{pmatrix}m-l \\ k-l \end{pmatrix}\begin{pmatrix}n-k+l \\ l \end{pmatrix}\\
&= \begin{pmatrix}m+n-2k \\ m-k \end{pmatrix} \sum_{l=0}^{k} (-1)^l  \begin{pmatrix}k\\ l \end{pmatrix}\begin{pmatrix}n-k+l \\ k \end{pmatrix}
\end{align*}
The proposition then  follows from (4.11) with $r=k$ and $s=n-k.$
\end{proof}

\begin{theorem}[Clebsch-Gordan Coefficients]
\begin{equation} C_{m, n, k}(i, j) = \frac{(-1)^k}{D(m,\, n,\, k)}\ \ c_{m, n, k}(m-i, n-j)\end{equation}
where 
\begin{align*} D(m,\, n,\, k) &=  \begin{pmatrix}m+n-k+1 \\ k \end{pmatrix} \ \begin{pmatrix}m+n-2k \\ m-k \end{pmatrix}\\ 
&= \frac{m+n-k+1}{m+n-2k+1} \begin{pmatrix}m+n-k \\ m-k,\ n-k,\ k \end{pmatrix}
\end{align*}
\end{theorem}

\begin{proof} The invariant bilinear form on $V(m)\otimes V(n)$ gives
\begin{equation*}\langle f^i\phi_m\otimes f^j\phi_n, f^{m-i}\phi_m\otimes f^{n-j}\phi_n\rangle = (-1)^{i+j},\end{equation*}
and
\begin{equation*}\langle f^i\phi_{m,n,k}, f^{m+n-2k-i}\phi_{m,n,k}\rangle = (-1)^{i} D(m,n,k),\end{equation*}
where $D(m,n,k)$ depends only on $m, n,$ and $k.$  To compute $D(m,n,k)$, first note
\begin{align*} (-1)^{i+j} c_{m,n,k}(i,j) &= \langle f^i\phi_m\otimes f^j\phi_n, f^{m+n-i-j-k}\phi_{m,n,k}\rangle \\ &= (-1)^{i+j-k} D(m,n,k)\ C_{m,n,k}(m-i, n-j) \end{align*}

Substituting into (4.5) gives
\begin{equation}\sum\limits_{i+j=p}\ c_{m,n,k}(i, j)\ c_{m,n,k}(m-i, n-j) = (-1)^k\ D(m,n,k)\end{equation}
Since this sum is independent of $p$, choose $p=k.$ That is, one dots the coordinate vector of the highest weight vector with the reverse of the lowest weight vector.
Using Proposition 4.9, the dot product has the form
\begin{equation*}(-1)^k\ \begin{pmatrix}m+n-2k \\ m-k \end{pmatrix}\sum_{l=0}^k\begin{pmatrix}m-l \\ k-l \end{pmatrix}\ \begin{pmatrix}n-k+l \\ l \end{pmatrix}.\end{equation*} The terms in the sum are the coordinates of the highest weight vector without signs.  In turn, the sum occurs as a term in the  convolution of binomial series; one obtains 
\begin{equation*}\sum_{r+s=k}\begin{pmatrix}m-r \\ m-k \end{pmatrix}\ \begin{pmatrix}n-s \\ n-k \end{pmatrix} =    \begin{pmatrix}m+n-k+1 \\ k \end{pmatrix}.\end{equation*}

\end{proof}

Finally, substituting into Proposition 4.4 gives

\begin{corollary}[Orthogonality Relations 2]
\begin{equation}\sum\limits_{i+j=p}\ c_{m,n,k}(i, j)\ c_{m,n,k'}(m-i, n-j) = (-1)^{k}\ \delta_{k, k'}\ D(m, n, k)\end{equation}
\begin{equation}\sum\limits_{k=0}^m\ (-1)^k\ C_{m,n,k}(m-i, n-j)\ C_{m, n, k}(i', j')\ D(m,n,k) =\ \delta_{i,i'}\delta_{j, j'}\end{equation}
where $i+j=i'+j'.$

\end{corollary}

\section{Example}

The matrices in Example 5.1 follow from the methods of Section 4. The matrices on the left-hand side represent coordinates for weight vectors. Weights are in descending order, decreasing by 2 at each column. The highest weight vector is represented by the leftmost column (Proposition 4.1), and one moves to the right using Pascal's recurrence (uppercase L pattern). When $k=0,$ one obtains part of Pascal's Triangle.

The matrices on the right-hand side represent Clebsch-Gordan coefficients.  Each column gives coordinates for $f^i\phi_m\otimes f^j\phi_n$ in terms of the corresponding weight vectors for each $V(m')$-type.

Consider the third columns for the  $V(5)$-component in Example 5.1.  One uses the basis
\begin{equation}\{\phi_3\otimes f^3\phi_4,\ \ \ f\phi_3\otimes f^2\phi_4,\ \ \ f^2\phi_3\otimes f^1\phi_4,\ \ \ f^3\phi_3\otimes \phi_4\},\end{equation}
since the weight is $5-2(2)=1$  The corresponding coordinate vector is $ (3 , 2, -5, -4),$ and 
\begin{equation*}v=f^2\phi_{3,4,1} = 3\ \phi_3\otimes f^3\phi_4\ +\ 2\ f\phi_3\otimes f^2\phi_4\ - 5\ f^2\phi_3\otimes f\phi_4\ -4\ f^3\phi_3\otimes \phi_4.\end{equation*}

The Clebsch-Gordan coefficients for weight $1$ are $\frac{-1}{70}(-9,-3, 5, 3)$, and
\begin{equation*}P_{V(5)}(\phi_3\otimes f^3\phi_4)=\frac{9}{70}\ v,\quad P_{V(5)}(f\phi_3\otimes f^2\phi_4)=\frac{3}{70}\ v, \quad \text{etc.,}\end{equation*}
where $P_{V(5)}$ records the component in the $V(5)$-type.

\begin{example}
\begin{equation*} k=0:\  V(7) \subset V(3)\otimes V(4)\end{equation*}
\begin{equation*}\begin{bmatrix}
1 & 1 & 1 & 1 & 1 & 0 & 0 & 0\\
0 & 1 & 2 & 3 & 4 & 5 & 0 & 0 \\
0 & 0 & 1 & 3 & 6 & 10 & 15 & 0\\
0 & 0 & 0 & 1 & 4 & 10 & 20 & 35
\end{bmatrix};\qquad\frac{1}{35}\begin{bmatrix}
 35 & 20 & 10 & 4 & 1 & 0 & 0 & 0\\
 0 & 15 & 10 & 6 & 3 & 1 & 0 & 0\\
 0 & 0 & 5 & 4 & 3 & 2 & 1 & 0\\
 0 & 0 & 0 & 1 & 1 & 1 & 1 & 1
\end{bmatrix}
\end{equation*}

\begin{equation*}k=1:\ V(5) \subset V(3)\otimes V(4)\end{equation*}
\begin{equation*}\begin{bmatrix}
\ \ 3 & \ \ 3 & \ \ 3 & \ \ 3 & \ \ \ \ 0 & \ \ \ 0\\
-4 & -1 & \ \ 2 &\ \ 5 &\ \ \ \ 8 &\ \  \ 0\\
\ \ 0 & -4 & -5 & -3 &\ \  \ \ 2 &\ \ 10\\
\ \ 0 & \ \ 0 & -4 & -9 & -12 & -10
\end{bmatrix};\quad\frac{-1}{\ 70}\begin{bmatrix}
-10 & -12 & -9 & -4 & \ \ 0 & \ \ 0\\
 \ \ 10 & \ \ 2 & -3 & -5 & -4 & \ \ 0\\
 \ \ 0 & \ \ 8 & \ \ 5 & \ \ 2 & -1 & -4\\
\ \ 0 & \ \ 0 & \ \ 3 &\ \  3 &\ \  3 &\ \  3
\end{bmatrix}
\end{equation*}

\begin{equation*}k=2:\ V(3) \subset V(3)\otimes V(4)\end{equation*}
\begin{equation*}\begin{bmatrix}
\ \ 3 &\ \  3 &\ \  3 & \ \ 0\\
-6 & -3 & \ \ 0 &\ \ 3 \\
\ \ 6 & \ \ 0 & -3 & -3 \\
\ \ 0 & \ \  6 &\ \  6 & \ \ 3
\end{bmatrix};\qquad\frac{1}{45}\begin{bmatrix}
  \ \ 3 &\ \  6 &\ \  6  & \ \  0\\
 -3 & -3 & \ \ 0 &\ \  6\\
 \ \ 3 & \ \ 0 & -3 & -6\\
 \ \ 0 & \ \ 3 & \ \ 3 & \ \ 3 
\end{bmatrix}
\end{equation*}

\begin{equation*}k=3:\ V(1) \subset V(3)\otimes V(4)\end{equation*}
\begin{equation*}\begin{bmatrix}
\ \ 1 & \ \ 1\\
-2 & -1\\
\ \ 3 & \ \ 1 \\
-4 & -1
\end{bmatrix};\qquad\frac{-1}{\ 10}\begin{bmatrix}
-1 & -4\\
\ \ 1 &\ \  3\\
-1 & -2\\
 \ \ 1 & \ \ 1 
\end{bmatrix}\end{equation*}
\end{example}

\section{Algorithm with Pascal's recurrence}

Section 5 provides a template for computer output as a polytope of values, with each horizontal slice representing an irreducible constituent. In terms of modeling, the methods of Section 4 enable a holistic alternative to the traditional approach using summation formulas and give a recursive method using only integer types.  For sophisticated models that compute single summations with large integers, see, for instance, \cite{JF}.  

\begin{theorem}[Algorithm with Pascal's recurrence]  To calculate the coordinate vector matrix for $V(m+n-2k)$ in $V(m)\otimes V(n)$:

\begin{enumerate}
\item initialize an $(m+1)$-by-$(m+n-2k+1)$ matrix,
\item set up coordinates for the highest weight vector in the leftmost column using Proposition 4.1, and extend the top row value,
\item apply Pascal's recurrence rightwards in an uppercase L pattern, extending by zero where necessary,
\item for the zero entries in lower-left corner, corresponding entries in the upper-right corner are set  to zero, and
\item the $(i+1, i+j-k+1)$-th entry is $c_{m,n,k}(i, j).$
\end{enumerate}

To calculate the Clebsch-Gordan coefficients associated to the corresponding basis vectors:

\begin{enumerate}
\item rotate the coordinate vector matrix by $180$ degrees,
\item divide by the normalizing factor $(-1)^k D(m, n, k)$ in (4.12), and
\item the $(i+1, i+j-k+1)$-th entry is $C_{m,n,k}(i, j).$
\end{enumerate}
\end{theorem}

As a check on the algorithm, the dot product of vectors in corresponding columns from these matrices always equals one, as is easily verified in Example 5.1.

\section{Recurrence Relations}

Several families of recurrences and symmetries apply to $C_{m,n,k}(i, j)$.  These are often expressed as consequences of special function theory, in particular, hypergeometric series of type ${}_3F_2.$  In the present setting, these may be recaptured through representation theory and binomial coefficient properties, in particular though generating functions or the summation formula.  For the remainder, results are stated for $c_{m,n,k}(i, j)$, noting that those for $C_{m,n,k}(i, j)$ are simply recaptured using Corollary 8.5. 

The first two recurrences follow by applying $f$ or $e$ to basis vectors. Compare with \cite{Vi}, Ch. III., Sec. 8.7, (4) and (5).  Alternatively,  Proposition 7.2 follows from applying (8.5) to (7.1).

\begin{proposition}[Pascal's Recurrence]

\begin{equation}c_{m, n, k}(i, j) = c_{m, n, k}(i, j-1) + c_{m, n, k}(i-1, j)\end{equation}
\end{proposition}

\begin{proposition}[Reverse Recurrence]
\begin{multline}\notag (i'+1)(m+n-2k-i')\ c_{m, n, k}(i,j)=\\ (i+1)(m-i)\ c_{m, n, k}(i+1, j) + (j+1)(n-j)\ c_{m, n, k}(i, j+1)\end{multline}
where $i'=i+j-k$.
\end{proposition}

The next set of recurrences follow from the generating function for $c_{m,n,k}(i, j)$ and Proposition 7.1.  
\begin{proposition}[Outer Recurrences]
\begin{equation}c_{m, n, k}(i, j) = c_{m-1, n, k}(i, j)  + \ c_{m-1, n-1, k-1}(i, j-1)\end{equation}
\begin{equation}c_{m, n, k}(i, j) = c_{m, n-1, k}(i, j) -  \ c_{m-1, n-1, k-1}(i-1, j)\end{equation}
\begin{equation}c_{m, n, k}(i, j) = \ c_{m-1, n, k}(i-1, j) + \ c_{m, n-1, k}(i, j-1)\end{equation}
\begin{equation}c_{m, n, k}(i, j) = \ c_{m+1, n, k+1}(i, j) - \ c_{m, n+1, k+1}(i, j).\end{equation}
\end{proposition}

\begin{proof}   (7.2) follows by noting that
\begin{equation*}(1-x)\cdot \frac{(x+y)^{i+j-k}}{(1-x)^{m-k+1}(1+y)^{n-k+1}} = \frac{(x+y)^{i+j-k}}{(1-x)^{m-k}(1+y)^{n-k+1}}.\end{equation*}
(7.3) follows in a similar manner, instead using the factor $(1+y)$.

For (7.4), subtract 1 from $i$ and $j$ in (7.1) and (7.2), respectively.  Adding both equations gives
\begin{equation*} c_{m, n, k}(i-1, j) + c_{m, n, k}(i, j-1) = c_{m-1, n, k}(i-1, j) + c_{m, n-1, k}(i, j-1),\end{equation*}
and the result follows from Proposition 7.1.

The proof of (7.5) follows in a similar matter with subtraction after increasing $m$, $n$ and $k$ by 1.
\end{proof}

Alternatively, Proposition 7.1 and (7.2) follow directly from Proposition 4.7 and Pascal's identity
\begin{equation}\begin{pmatrix} r \\ s\end{pmatrix} = \begin{pmatrix} r-1 \\ s\end{pmatrix}  +  \begin{pmatrix} r-1 \\ s-1\end{pmatrix}.\end{equation}

\section{Regge Symmetries}

In the 1950s, Regge \cite{Re} added to known symmetries to obtain a group of order 72 associated to $C'_{m,n,k}(i, j)$ in the classical context.  See \cite{Ko}, \cite{Kr}, or \cite{Vi}, Ch. III, Sec. 8.4 for detailed background.  These symmetries are efficiently encoded by transfer of indices in accordance with the following matrix, called the Regge symbol:
\begin{equation}\begin{Vmatrix}\ \  n-k\quad &  m-k & k \\ \ \ i & j &\ \  m+n-i-j-k\ \\ \ \ m-i \quad & n-j & i+j-k \end{Vmatrix}.\end{equation}

Symmetries correspond to the usual symmetries for determinant; one may permute any rows, permute any  columns, or perform matrix transpose.  Thus, to describe all symmetries in this form, it is enough to describe a generating set of row switches and transpose. The column switch corresponding to tensor transpose is also included. Table 1 summarizes index transformations.

With the conventions above, these generating symmetries are given below.  For abbreviations, for instance, R13 indicates the interchange of rows 1 and 3 in the Regge symbol.

\begin{proposition}[C12 Symmetry - Tensor Transpose]
\begin{equation}c_{n,m,k}(j, i) = (-1)^kc_{m, n, k}(i, j)\end{equation}
\end{proposition}

\begin{proposition}[Modified Transpose Symmetry: $Transpose\circ R23$]
\begin{equation}\begin{pmatrix}m-k+j \\ m-k\end{pmatrix} \ c_{m,\, n-k+i,\, i}(k,\ i+j-k) =\begin{pmatrix}m-k+j \\ m-i \end{pmatrix} c_{m, n, k}(i, j)\end{equation}
\end{proposition}

\begin{proposition}[R13 Symmetry]
\begin{equation}\begin{pmatrix}m+n-k \\ m-k,\ n-k,\ k \end{pmatrix}\ c_{m', n', k'}(i, j) = (-1)^i \begin{pmatrix} m+n-k \\ m-i,\ n-j,\ k' \end{pmatrix} c_{m, n, k}(i, j)\end{equation}
where $m'=n-k+i,\ n'=m-k+j,\ $ and $k'=i+j-k$.
\end{proposition}

\begin{proposition}[R23 Symmetry - Weyl Group]

\begin{equation}\begin{pmatrix}m+n-k \\ i,\ j,\ * \end{pmatrix}  c_{m, n, k}(m-i, n-j) = (-1)^k \begin{pmatrix}m+n-k \\ m-i,\ n-j,\ * \end{pmatrix} c_{m, n, k}(i, j)\end{equation}
\end{proposition}

\begin{table}
\caption{Regge Symmetries}
\begin{tabular}{|c|c|c|c|c|c|}
\hline
  & m & n & i & j  & k \\
 \hline
C12 & $n$ & $m$ & $j$ & $i$ & $k$ \\
Transpose & $m$ & $m+n-i-k$ & $m-k$ & $j$ & $m-i$ \\
R13 & $n-k+i$ & $m-k+j$ & $i$ & $j$ & $i+j-k$ \\
R23 & $m$ & $n$ & $m-i$ & $n-j$ & $k$ \\
\hline
\end{tabular}
\end{table}

\begin{proof} The C12 symmetry follows directly from the definition of the coordinate vector for highest weights.  The R23 symmetry follows from the Weyl group action
\begin{equation}w(f^p\phi_m)=(-1)^p \frac{p!}{(m-p)!} f^{m-p}\phi_m\end{equation}
where $w$ is the group element 
\begin{equation}w=\begin{bmatrix}0 & -1 \\ 1 & \ \ 0 \end{bmatrix}.\end{equation}
The R13 and modified Transpose symmetries follow from direct substitution into Proposition 4.7. The modified Transpose symmetry follows immediately.  For the R13 symmetry, one also changes the summation variable to $l' = i-l.$
\end{proof}

The R23 Symmetry and Theorem 4.10 immediately yield

\begin{corollary} 
\begin{equation} C_{m, n, k}(i, j) = \frac{\begin{pmatrix}m+n-k \\ m-i,\ n-j,\ * \end{pmatrix}\  c_{m, n, k}(i, j)}{\begin{pmatrix}m+n-k \\ i,\ j,\ *\end{pmatrix}\ D(m, n, k)}\end{equation}
\end{corollary}

\section{Normalization}

Another feature of the theory concerns alternative expressions for Clebsch-Gordan coefficients, either in terms of binomial summations \cite{Sh}  or hypergeometric series \cite{Ra}.  For the most part, these expressions are unified through the Regge group, which in turn is essentially a mechanism from hypergeometric theory.  Due to their ubiquity in the literature, brief derivations of Wigner's and Racah's formulas are given here.

\begin{proof}{(Proposition 3.2)}
Extending coefficients from the $SL(2, \real)$ theory with real vector spaces, one may choose $\phi_m$, and in turn all $f^i\phi_m$, to be real in $V(m)$.  With this choice of basis for $V(m),$ one restricts the group action to 
\begin{equation*}SU(2)=\{g\in SL(2, \cpx)\ |\ g\cdot g^*=I_2\},\end{equation*}
where $g^*$ is the conjugate transpose of $g$ and $I_2$ is the identity matrix of size 2. Note that
\begin{equation}w = \begin{bmatrix} 0 &-1 \\ 1 &\ \ 0 \end{bmatrix}\end{equation}
is in $SU(2),$ and, for $g$ in $SU(2)$,
\begin{equation}wgw^{-1} = (g^{-1})^T = {\bar g}.\end{equation}
For $v_1, v_2$ in $V(m)$, define the invariant Hermitian inner product
\begin{equation}\langle\langle v_1, v_2\rangle\rangle_m = \langle v_1, \overline{w\cdot v_2}\rangle_m.\end{equation}
Then
\begin{align}  \langle\langle f^i\phi_m, f^i\phi_m\rangle\rangle_m &= \langle f^i\phi_m, \overline{w\cdot f^i\phi_m}\rangle_m\\
&=(-1)^i\ \frac{i!}{(m-i)!}\langle f^i\phi_m, \overline{f^{m-i}\phi_m}\rangle_m\notag\\
&=(-1)^i\ \frac{i!}{(m-i)!}\langle f^i\phi_m, f^{m-i}\phi_m\rangle_m = \frac{i!}{(m-i)!}\notag
\end{align}

Likewise, in $V(m)\otimes V(n)$,
\begin{equation}\langle\langle f^{i+j-k}\phi_{m,n,k}, f^{i+j-k}\phi_{m,n,k}\rangle\rangle =  \frac{(i+j-k)!\ D(m, n, k)}{(m+n-i-j-k)!}.\end{equation}

Normalizing basis vectors, one recovers Wigner's Formula as
\begin{equation}C'_{m,n,k}(i, j) = \frac{||f^{i+j-k}\phi_{m,n,k}||}{||f^i\phi_m||_m\ ||f^j\phi_n||_n} C_{m,n,k}(i, j).\end{equation}
\end{proof}

Next one has Racah's formula

\begin{proposition}[\cite{Vi}, Ch. III.8.9, Equation (10)]
 
 {\fontsize{10}{10}\selectfont
\begin{equation*}C'_{m,n,k}(i, j)\ =\sqrt{  \frac{(m+n-2k+1)\begin{pmatrix} m+n-k \\ m-k,\ n-k,\ k\end{pmatrix} }{{(m+n-k+1)\ \begin{pmatrix} m+n-k \\ m-i,\ n-j,\ *\end{pmatrix}  \begin{pmatrix} m+n-k \\ i,\ j,\ *\end{pmatrix}}}}\cdot S'
\end{equation*}
}
where 
$$S'=\sum\limits_{s=0}^k \ (-1)^{s} \begin{pmatrix} k \\ s\end{pmatrix}\begin{pmatrix} m-k \\ i-s\end{pmatrix}\begin{pmatrix} n-k \\ j-k+s\end{pmatrix}.$$
\end{proposition}

Comparing with Wigner's Formula, Proposition 9.1 follows immediately from

 \begin{proposition}
\begin{equation*}c_{m,n,k}(i,j)=\sum\limits_{l=0}^k \ (-1)^{l} \begin{pmatrix} i+j-k \\ i-l\end{pmatrix}\begin{pmatrix} m-i \\ k-l\end{pmatrix}\begin{pmatrix} n-j \\ l\end{pmatrix}.\end{equation*}
\end{proposition}

\begin{proof}[Proof by induction] First note that, for any $m$ and $n$ with $k=0,$ only one term contributes in either sum; that is,
\begin{equation*}c_{m,n,0}(i, j) = \begin{pmatrix} i+j \\ i\end{pmatrix}.\end{equation*}

The induction is indexed by $i'=i+j-k$ (all $m, n$).  For the base case, when $i'=0,$ the first binomial coefficient only contributes when $l=i$, and the formula  agrees with Proposition 4.1.

For the induction step, suppose the proposition holds for $i'=i+j-k-1.$ Using (7.6), one has
\begin{multline*}
\sum\limits_{l=0}^k \ (-1)^{l} \begin{pmatrix} i+j-k \\ i-l\end{pmatrix}\begin{pmatrix} m-i \\ k-l\end{pmatrix}\begin{pmatrix} n-j \\ l\end{pmatrix}\\
=\sum\limits_{l=0}^k \ (-1)^{l} \bigg[\begin{pmatrix} i+j-k-1 \\ i-l-1\end{pmatrix}+\begin{pmatrix} i+j-k-1 \\ i-l\end{pmatrix}\bigg]\begin{pmatrix} m-i \\ k-l\end{pmatrix}\begin{pmatrix} n-j \\ l\end{pmatrix}\\
= c_{m-1,n,k}(i-1, j) + c_{m,n-1,k}(i, j-1) = c_{m,n,k}(i, j).\\
\end{multline*}
The last equality holds by the recurrence relation (7.4).
\end{proof}

\section{Appendix: Projection Operator Methods}

The formulas in Section 4 were discovered after explicitly computing projection operators for the types $V(m')$ in $V(m)\otimes V(n)$.  These methods depend on the multiplicity-free properties of both the weight vectors in $V(m')$ and the Clebsch-Gordan decomposition.  Extremal projector methods for $SU(2)$ first appear in  \cite{Lo};  see \cite{To} for background on extremal projectors. An algorithm for constructing these projection operators follows.

Define the Casimir element of $U(sl(2))$ by
\begin{equation}8\Omega \ = \ 2ef + 2fe + h^2\ = \ 4ef + h^2 - 2h.\end{equation}
This element lies in the center of $U(sl(2))$. In fact,  $Z(U(sl(2)) = \mathbb{C}[\Omega].$
One notes two useful facts about $\Omega$:
\begin{enumerate}
\item On $V(n)$, $8\Omega$ acts by the scalar $\lambda_n=n(n+2)$
\item  $\Omega$ preserves weight spaces:   if $hv=cv$ then $h(\Omega v) = c (\Omega v)$
\end{enumerate}
For computational efficiency, one may use $ef$ instead of $\Omega$, as $ef$ also preserves weight spaces.

Recall spectral theory for finite-dimensional vector spaces.  Let $M$ be an $r\times r$ matrix,  diagonalizable with distinct eigenvalues $\lambda_i,\ 1\le i \le r.$  If one computes the  projection operator for the eigenspace $W(\lambda_j)$ of $M$
then $W(\lambda_j)$ is spanned by any nonzero column of the projection operator.
Fix $1\le j \le r.$   For $i\ne j$, consider the quotients
\begin{equation}M_i(\lambda_j) = \frac{M-\lambda_i I}{\lambda_j - \lambda_i}\end{equation}
Then
\begin{equation}M_i(\lambda_j) = \begin{cases} 1 & \text{on \ } W(\lambda_j)\\ 0 & \text{on\ } W(\lambda_i)\\  * & \text{on other \ }W(\lambda_k)  \end{cases}\end{equation}
Now define
\begin{equation}P(\lambda_j) = \prod_{i\ne j} M_i(\lambda_j)\ : \ \cpx^r \rightarrow W(\lambda_j) \subset \cpx^r.\end{equation}

Fix $V(m)$ and $V(n),$ and  define 
\begin{equation}M=[ef]_{\bf B_p},\end{equation}
where ${\bf B_p}$ is the basis for the weight space for $m+n-2p$ in $V(m)\otimes V(n).$  As noted, $M$ is diagonalizable with eigenvalues of multiplicity one.  In fact, $ef$ acts on the weight vector in $V(m+n-2k)$ with eigenvalue
\begin{equation}\lambda_{m, n, k, p} = (m+n-p-k)(p-k+1).\end{equation}

Next, using formulas from Section 2,  the effect of $ef$ on the elements of ${\bf B}_p$ is given by 
\begin{proposition}[Tridiagonal Formula]
\begin{align} (ef)(f^{i}\phi_m\otimes f^{j}\phi_n)&= (ef^{i+1}\phi_m)\otimes f^{j}\phi_n\ \ + \ \ f^{i}\phi_m\otimes (ef^{j+1}\phi_n)\\
								&\quad + (ef^{i}\phi_m)\otimes f^{j+1}\phi_n\ \ +\  \  f^{i+1}\phi_m\otimes (ef^{j}\phi_n)\notag\end{align}
\begin{align*} \quad				&= [(i+1)(m -i)\ + \ (j+1)(n-j)]\, f^{i}\phi_m\otimes f^{j}\phi_n\\
								&\qquad\qquad\qquad\qquad\  + \ i(m-i+1)\ f^{i-1}\phi_m\otimes f^{j+1}\phi_n\\ &\qquad\qquad\qquad\qquad +\ \ j(n-j+1)\, f^{i+1}\phi_m\otimes f^{j-1}\phi_n\end{align*}
\end{proposition}			
								
\begin{example}  Returning to Example 5.1, consider the subspace for weight 1 in $V(3)\otimes V(4)$, with basis
\begin{equation}B_3= \{\phi_3\otimes f^3\phi_4,\ f\phi_3\otimes f^2\phi_4,\ f^2\phi_3\otimes f\phi_4,\,\ f^3\phi_3\otimes \phi_4\},\end{equation}
Using the Tridiagonal Formula, $ef$ acts on the this subspace in coordinates by 
\begin{equation}M=[ef]_{B_3} = \begin{bmatrix} 7 & 3 & 0 & 0\\ 6 & 10 & 4 & 0\\ 0 & 6 & 9 & 3\\ 0 & 0 & 4 & 4\end{bmatrix}\end{equation}
On each $V(7-2k)$ with $k \in \{3, 2, 1, 0\},$ $M$ has respective eigenvalues $\{1, 4, 9, 16\}$.
Using the projection operator formula, the corresponding projections are given by
\begin{equation*}P_{V(1)}=\frac{1}{10}\begin{bmatrix}\ \ 1 & -1 & \ \ 1 & -1\\ -2 &\ \  2 & -2 &\ \ 2\\ \ \ 3 & -3 & \ \ 3 & -3\\ -4 & \ \ 4 & -4 & \ \ 4\end{bmatrix},\ P_{V(3)}=\frac{1}{5}\begin{bmatrix} \ \ 2 & -1 &\ \ \  0 & \ \ 1\\ -2 & \ \ 1 & \ \ \ 0 & -1\\ \ \   0 &\ \   0 &\ \ \  0 &\ \   0\\ \ \ 4 & -2 & \ \ \ 0 & \ \ 2\end{bmatrix},\end{equation*} \begin{equation*}P_{V(5)}=\frac{1}{70}\begin{bmatrix} \ \ 27 & \ \ \ 9 & -15 & -9\\ \ \ 18 & \ \ \ 6 & -10 & -6\\ -45 & -15 & \ \ \ 25 & \  15\\ -36 & -12 & \ \ \ 20 & \  12\end{bmatrix},\  P_{V(7)}=\frac{1}{35}\begin{bmatrix}\ 4 &\  6 & \ 4 & \ 1\\  12 & 18 & 12 &\  3\\ 12 & 18 & 12 & \ 3\\ \ 4 & \ 6 & \ 4 &\  1\end{bmatrix}.\end{equation*}

Consider the basis vector $f^2\phi_3\otimes f\phi_4$  from ${\bf B}_3$ in $V(3)\otimes V(4)$.
In coordinates,  $[f^2\phi_3\otimes f\phi_4]_{\bf B_3} = (0, 0, 1, 0).$  The decomposition according to type follows from the third columns in the projection operators:
\begin{align}  (0,0, 1, 0)  &= (1/10, -1/5 , 3/10 , -2/5)  \tag{V(1)}\\
&+ (0,\quad \quad 0,\quad \quad 0,\quad \quad 0)\tag{V(3)}\\
&+ (-3/14, -1/7, 5/14, 2/7)\tag{V(5)}\\
&+ (4/35, 12/35, 12/35, 4/35)\tag{V(7)}
\end{align}

Finally the coordinate vectors and Clebsch-Gordan coefficients in Section 5 factor the projection operators.  For instance,

\begin{equation*}P_{V(5)} = \begin{bmatrix} 3 & 0 &\ \ 0 &\ \  0\\  0 &  2 &  \ \ 0 & \ \  0\\ 0 & 0 & -5 & \ \ 0\\ 0 &   0 &\ \  0 & -4\end{bmatrix}\begin{bmatrix}  1 & 1 & 1 & 1\\  1 & 1 &  1 & 1\\ 1 & 1 & 1 & 1\\ 1 & 1 & 1 & 1\end{bmatrix}\begin{bmatrix} 9/70 & 0 & 0 & 0\\ 0 & 3/70 & 0 & 0\\ 0 & 0 & -5/70 & 0\\ 0 & 0 & 0 & -3/70\end{bmatrix}\end{equation*}
\end{example}

\bibliographystyle{amsplain}

\end{document}